\documentclass[12pt]{amsart}
\usepackage{amsfonts, amsmath, amsthm}

\usepackage{graphicx}
\usepackage{psfrag}

\swapnumbers

\newtheorem{pro}{Proposition} 
\newtheorem{thm}[pro]{Theorem}

\newtheorem{clm}[pro]{Claim}

\theoremstyle{definition}
\newtheorem{dfn}[pro]{Definition}

\theoremstyle{remark}

\newcommand{\bdy}{\partial}

\title{On the existence of high index topologically minimal surfaces}
\date{\today}
\address{Pitzer College}
\email{bachman@pitzer.edu}
\author{David Bachman$^1$}
\thanks{$^1$Partially supported by NSF grant DMS-0906151}

\address{Oklahoma State University}
\email{jjohnson@math.okstate.edu}
\author{Jesse Johnson$^2$}
\thanks{$^2$Partially supported by NSF MSPRF grant 0602368}

\begin{document}

\begin{abstract}
The topological index of a surface was previously introduced by the first author as the topological analogue of the index of an unstable minimal surface. Here we show that surfaces of arbitrarily high topological index exist.
\end{abstract}

\maketitle

\markleft{EXISTENCE OF HIGH INDEX TOPOLOGICALLY MINIMAL SURFACES}
\markright{DAVID BACHMAN AND JESSE JOHNSON}



Consider a compact, connected, two sided surface  $S$ properly embedded in a compact, orientable 3-manifold $M$.  The \textit{disk complex} $\Gamma(S)$ is the simplicial complex defined as follows:  Vertices of $\Gamma(S)$ are isotopy classes of compressing disks for $S$. A collection of $n$ such isotopy classes is an $(n-1)$-simplex of $\Gamma(S)$ if there are representatives of each that are pairwise disjoint. 

\begin{dfn}
\label{d:Indexn}
If $\Gamma(S)$ is non-empty then the {\it topological index} of $S$ is the smallest $n$ such that $\pi_{n-1}(\Gamma(S))$ is non-trivial. If $\Gamma(S)$ is empty then $S$ will have topological index 0.  If $H$ has a well-defined topological index (i.e. $\Gamma(S)=\emptyset$ or some homotopy group of $\Gamma(S)$ is non-trivial) then we will say that $S$ is {\it topologically minimal}. 
\end{dfn}

Topological index was introduced by the first author as the topological analogue of the index of an unstable minimal surface \cite{TopIndexI}, and later used in \cite{StabilizationResults} and \cite{AmalgamationResults} to prove various results about Heegaard splittings of 3-manifolds.  Here we show this definition is not vacuous for high index surfaces by proving the following existence result:

\begin{thm}
\label{t:ExistenceTheorem}
There is a closed 3-manifold, $M^1$, with an index 1 Heegaard surface $S$, such that for each $n$, the lift of $S$ to some $n$-fold cover $M^n$ of $M^1$ has topological index $n$. 
\end{thm}

The manifold $M^1$ of Theorem~\ref{t:ExistenceTheorem} is obtained by gluing together the boundary components of the complement of a link in $S^3$, which we will construct as follows:

We say $S^2 \subset S^3$ is a {\it bridge sphere} for a knot or link $L \subset S^3$ if $L$ meets each of the balls bounded by $S^2$ in a collection of boundary parallel arcs. If the minimum number of such arcs is $b$, then we say $L$ is a {\it $b$-bridge knot/link.}

Throughout the paper, we will assume $L \subset S^3$ is a two component two-bridge link such that for a regular neighborhood $N$ of $L$ the complement $M = S^3 \setminus N$ contains no essential planar surface with Euler characteristic greater than $-3$.

By \cite{ht:85} or \cite{bridge}, such links can be constructed by choosing a sufficiently complicated braid to define the link. Let $S^2$ be a bridge sphere for $L$ that realizes its bridge number, and let $B^\pm$ be the balls in $S^3$ bounded by $S^2$. Let $S = (S^2 \setminus N) \subset M$ and $C^\pm=(B^\pm \setminus N) \subset M$. The manifold $M^1$ of Theorem~\ref{t:ExistenceTheorem} is obtained from $M$ by gluing its boundary components together in such a way so that the surface $S$ glues up to a closed, orientable surface. 

A \textit{boundary compressing disk} for $S$  in $M$ is a disk $D$ with interior disjoint from $S$ and $\partial M$ such that $\partial D$ consists of an essential arc in $S$ and an arc in $\partial M$.  Because the interior of $D$ is disjoint from $S$, the disk $D$ is contained in either $C^-$ or $C^+$.  

\begin{clm}
\label{c:BdyStrongIrreducibility}
Every compressing or boundary compressing disk for $S$ in $C^-$ meets every compressing or boundary compressing disk for $S$ in $C^+$.
\end{clm}

\begin{proof}
Assume for contradiction there are disjoint compressing or boundary compressing disks $E_0^- \subset C^-$ and $E_0^+ \subset C^+$ for $S$.  We will replace $E_0^-$ and $E_0^+$ with boundary compressing disks $E^-$, $E^+$ as follows: The intersection $C^- \cap \partial M$ consists of two annuli. If $E_0^-$ is a boundary compressing disk, then one of these two annuli contains an arc of $\partial E_0^-$. If this arc is essential in the annulus that contains it, then we will let $E^- = E_0^-$. 

If this is not the case, or if $E_0^-$ is a compressing disk, then $E_0^- \cap S$ is an arc with both endpoints in the same loop of $\partial S$ or a loop in $S$, respectively. In either case $E_0^- \cap S$ separates $S$ into two components, each of which is an annulus or a pair of pants.  For each of these components, there is a boundary compressing disk for $S$ in $C^-$ that intersects $\partial N \cap C^-$ in an essential arc and intersects $S$ in this component. One of these components contains the arc $E_0^+$.  We will let $E^-$ be the boundary compressing disk for the other component.

Thus we have chosen $E^-$ to be disjoint from $E_0^+$ and so that $E^- \cap \partial N$ is essential in $\partial N \cap C^-$.  A similar construction produces a boundary compressing disk $E^+$ in $C^+$ that is disjoint from $E^-$ and intersects $\partial N \cap C^+$ in an essential arc.

Recall that $M$ is the complement in $S^3$ of a regular neighborhood $N$ of $L$.  We can extend each disk $E^\pm$ to a disk $D^\pm$ in $S^3$ whose boundary consists of an arc in $L$ and an arc in $S^2$.  The disks $D^+$, $D^-$ will either be disjoint or they will intersect in one or two points contained in $L \cap S^2$.

If the two disks intersect in one point in $L \cap S^2$ then we will slide the two arcs of $L$ across these disks into $S^2$.  The resulting link $L'$ is isotopic to $L$ and intersects $S^2$ in a point and an arc.  We can isotope $L'$ so that the arc of intersection becomes a single point.  The resulting link has one bridge with respect to $S^2$, contradicting our assumption that $L$ is 2-bridge.

If the two disks intersect in two points in $L \cap S^2$ then we can again isotope the bridge arcs into $S^2$, so that the resulting link $L'$ is isotopic to $L$ and consists of a one-bridge component and a component contained in $S^2$.  Such a link is either a two component unlink or a Hopf link. In either case $M$ will contain an essential planar surface whose Euler characteristic is at least $-2$. 

Finally, if the two disks are disjoint then the frontier of a regular neighborhood of each $D^\pm$ in $B^\pm$ is a disk $A^\pm$. The loops $\partial A^+$ and $\partial A^-$ are disjoint in the four-punctured sphere $S^2 \setminus L$, and each loop separates two of the punctures from the other two, so they must be parallel.  If we isotope the disks so that their boundaries coincide, the resulting sphere separates $L$ into two one-bridge components, so $L$ is the two-component unlink. This is again a contradiction, as $M$ will contain an essential disk. This final contradiction implies that there are no disjoint disks $E_0^-$, $E_0^+$.
\end{proof}

The intersection of the bridge surface $S$ with $\partial M$ defines a meridianal slope in each component of $\partial M$.  We will also choose an arbitrary longitude for each boundary component.  For each $i \leq n$, let $M_i$ be a 3-manifold homeomorphic to $M$, and let $S_i$ and $C^\pm_i$ be the images of $S$ and $C^\pm$ in $M_i$. Define one component of $L$ to be the \textit{left component} and call the other the \textit{right component}.  Each $M_i$ then has a left and right boundary component inherited from the left and right component, respectively, of $L$.  The surface $S_i$ intersects each boundary component of $M_i$ in a pair of meridians, and each component of $\partial M_i$ inherits a special longitude from $M$.

Let $M^n$ be the result of gluing the right component of each $\partial M_i$ to the left component of $\partial M_{i+1}$, as well as gluing the right component of $\partial M_n$ to the left component of $\partial M_1$.  All gluings should be made so that meridians are sent to meridians and longitudes are sent to longitudes.  Moreover, we will glue so that the surfaces $S_i$ together form a closed surface $S^n \subset M^n$, with $\bigcup C^+_i$ a submanifold of $M$ whose boundary is $S^n$. Because we use the same gluing along each pair of tori, the manifold $M^n$ will be an $n$-fold cyclic cover of the manifold $M^1$. 

\begin{clm}
\label{itsahsurfacelem}
For each $n$, the surface $S^n \subset M^n$ is a genus $n+1$ Heegaard surface.
\end{clm}

\begin{proof}
The intersection of $L$ with $B^+$ is a pair of unknotted arcs. Hence, there is a disk $D^+ \subset B^+$ with $\partial D^+ \subset \partial B^+$ such that $D^+$ separates $B^+$ into two components, each containing an unknotted arc of $L \cap B^+$.  The disk $D^+$ can be chosen disjoint from $N$ so that its image in each $M_i$ is a disk $D^+_i \subset C^+_i$.  Because each component of $B^+ \setminus D^+$ contains an unknotted arc of $L \cap B^+$, the closure of the complement in $C^+_i$ of $D^+_i$ is a pair of solid tori. We will call these the {\it left} and {\it right} solid tori of $\overline{C_i^+ \setminus D^+_i}$, respectively as each meets the left and right boundary component of $M_i$. 

When we glue $M_i$ to $M_{i+1}$, the right solid torus of $\overline{C_i^+ \setminus D^+_i}$ is glued to the left solid torus of $\overline{C_{i+1}^+ \setminus D^+_{i+1}}$ along an annulus that is primitive in each. Hence, these two together form a solid torus, which we denote $T_i$. Similarly, the right solid torus of $\overline{C_n^+ \setminus D^+_n}$ is glued to the left solid torus of $\overline{C_1^+ \setminus D^+_1}$ to form a solid torus $T_n$. 

To reconstruct $\bigcup C_i^+$, for each $i <n$ we glue $T_i$  to $T_{i+1}$ along the disk $D^+_{i+1}$ and we glue $T_n$ to $T_1$ along $D^+_1$. Hence, $\bigcup C_i^+$ is a handlebody of genus $n + 1$.  A similar argument implies that $\bigcup C^-_i$ is also a handlebody.  The two handlebodies intersect along their common boundary $S^n$, so $S^n$ is a Heegaard surface for $M^n$.
\end{proof}

\begin{clm}
\label{nindexequalstoplem}
The surface $S^n$ has topological index at most $n$.
\end{clm}

\begin{proof}
As in the proof of Claim~\ref{itsahsurfacelem}, the Heegaard surface $S^n$ bounds handlebodies $\bigcup C^+_i$ and $\bigcup C^-_i$, and there is a compressing disk $D^+_i$ for $S^n$ contained in each $C^+_i$ and a second compressing disk $D^-_i$ with interior in $C^-_i$.  For $i \neq j$, the subsets $C^\pm_i$ and $C^\pm_j$ have disjoint interiors so $D^\pm_i$ and $D^\pm_j$ will be disjoint.  If $i = j$ then the disks are either the same or they are on opposite sides of $S_i$, so as noted above they must intersect.  In $\Gamma(S^n)$, there will thus be edges connecting the vertices corresponding to $D^+_i$ and $D^-_i$ to the vertices corresponding to $D^+_j$ and $D^-_j$, if and only if $i \ne j$. Thus the subset $P$ of  $\Gamma(S^n)$ spanned by these vertices is isomorphic to the dual complex of the boundary of an $n$-dimensional cube as in Figure \ref{f:octahedron}.

\begin{figure}[htbp]
\begin{center}
\psfrag{A}{$D_1^+$}
\psfrag{a}{$D_1^-$}
\psfrag{B}{$D_2^+$}
\psfrag{b}{$D_2^-$}
\psfrag{C}{$D_3^+$}
\psfrag{c}{$D_3^-$}
\includegraphics[width=2.5 in]{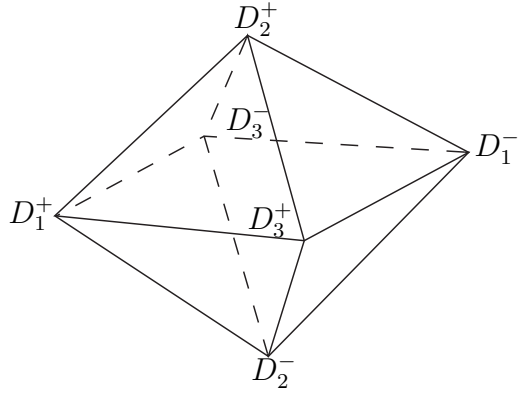}
\caption{The complex $P$ when $n=3$.}
\label{f:octahedron}
\end{center}
\end{figure}

We will define a map $F:\Gamma(S^n) \to P$ as follows:  Let $D$ be a compressing disk for $S^n$ and assume, without loss of generality, that $D$ is contained in $\bigcup C^+_i$.  Let $T$ be the image in $M$ of the torus boundary components of each $M_i$.  Assume that we have isotoped $D$ so as to minimize $\partial D \cap T$.  Because $T \cap S^n$ and $\partial D$ are essential in $S^n$, the result is canonical up to isotopy of $D \cup T$. (To see this, one can fix a hyperbolic metric on $S^n$ and isotope $T \cap S^n$ and $\partial D$ to their geodesic representatives in $S^n$.)  If $D$ is disjoint from $T$ then it is contained in the image of some $M_i$.  In this case, $F$ will send $D$ to $D_i^+$.  

Otherwise, note that $T$ is incompressible in $M$, so any innermost loop of $D \cap T$, bounding a disk in $D$, must also bound a disk in $T$.  Thus we can isotope $D$ so that $D \cap T$ is a collection of arcs.  Let $\alpha$ be an outermost arc which cuts off a disk $E \subset D$ whose interior is disjoint from $T$. Let $\beta \subset \partial D$ be the second arc making up $\partial E$.  If $\beta$ is boundary parallel in $S^n \setminus T$ then the isotopy of $\beta$ into $T$ determines an isotopy of $D$ that reduces $\partial D \cap T$.  (The isotopy produces a loop of $D \cap T$ that can then be removed.)  Thus we can assume that $\beta$ is essential in $S^n \setminus T$.  We will let $F$ send $D$ to $D^+_i$ for the smallest $i$ such that there is such an outermost subdisk $E$ of $D \setminus T$  lying in $C^+_i$.

We will show that $F$ sends the endpoints of each 1-simplex in $\Gamma(S^n)$ to the endpoints of a 1-simplex in $P$. Assume $D$ and $D'$ are disjoint compressing disks for $S^n$ such that $F(D)$ is either $D_i^+$ or $D_i^-$ and $F(D')$ is either $D_j^+$ or $D_j^-$.  If $i \neq j$ then the disks $D_i^\pm$, $D_j^\pm$ are disjoint and the edge from $D$ to $D'$ can be sent to the edge from $D_i^\pm$ to $D_j^\pm$, accordingly.  We need to show that if $i = j$ then $D$ and $D'$ must be on the same side of $S^n$, and thus are both sent to the same vertex of $P$.  

Assume for contradiction $D$ is on the positive side of $S^n$ and $D'$ is on the negative side.  By definition, after an appropriate isotopy fixing their boundaries, $D \cap M_i$ will be a compressing or $\bdy$-compressing disk $E$ for $S_i$ in $M_i$. Similarly, $D' \cap M_i$ will be a compressing or $\bdy$-compressing disk $E'$. As $E$ and $E'$ are on opposite sides of $S_i$, by Claim \ref{c:BdyStrongIrreducibility} they must intersect. But $E$ and $E'$ are subdisks of disjoint disks $D$ and $D'$.  This contradiction implies that $D$ and $D'$ are on the same side of $S^n$. 

Note that whenever the complex $P$ contains the boundary of an $n$-simplex, it also contains its interior. It follows that since $F$ is defined on the 1-skeleton of $\Gamma(S^n)$, and the higher dimensional cells in $\Gamma(S^n)$ are determined by its 1-skeleton, the map $F$ can be extended over the rest of $\Gamma(S^n)$. 

The subcomplex $P$ defines an $(n-1)$-sphere in $\Gamma(S^n)$. The map $F: \Gamma(S^n) \to P$ fixes $P$, and is thus a retraction of $\Gamma(S^n)$ onto an $(n-1)$-sphere. We conclude $\pi_{n-1}(\Gamma(S^n))$ is non-trivial, and thus the topological index of $S^n$ is at most $n$.
\end{proof}

\begin{clm}
\label{c:EulerChar}
The manifold $M$ does not contain any non-trivial, planar, topologically minimal surfaces whose Euler characteristic is at least $-1$. 
\end{clm}

\begin{proof}
Let $F$ be such a surface. It suffices to show that $F$ is essential, since we have assumed that the complement of $L$ does not contain any such surfaces. 

If $F$ is not essential then it must be compressible. If all compressing disks are on the same side of $F$, then by \cite{mccullough:91}, $\Gamma(F)$ is contractible, a contradiction. Furthermore, there must be vertices that represent compressions on opposite sides of $F$ that intersect, since otherwise $\Gamma(F)$ would be the join of two contractible complexes, and would thus be contractible. In particular, there must be a pair of essential loops on $F$ (the boundaries of these disks) that can not be isotoped to be disjoint. But any two essential loops on a planar surface with Euler characteristic at least $-1$ can be isotoped to be disjoint. The contradiction implies that $M$ does not contain any planar, topologically minimal surfaces of Euler characteristic is at least $-1$.
\end{proof}

To complete the proof, we will need the following Theorem, which is proved in~\cite{TopIndexI}.

\begin{thm}[Theorem 3.2 in~\cite{TopIndexI}]
\label{t:OriginalIntersection}
Let $F$ be a properly embedded, incompressible surface in an irreducible 3-manifold $M$. Let $S$ be a properly embedded surface in $M$ with topological index $n$. Then $S$ may be isotoped so that
	\begin{enumerate}
		\item $S$ meets $F$ in $p$ saddles, for some $p \le n$, and 
		\item the sum of the topological indices of the components of $S \setminus N(F)$ in $M \setminus N(F)$, plus $p$, is at most $n$. 
	\end{enumerate}
\end{thm}

\begin{clm}
The surface $S^n$ has topological index at least $n$.
\end{clm}

\begin{proof}
Assume for contradiction $S^n$ has topological index $k < n$.  As in the proof of Claim \ref{nindexequalstoplem}, let $T$ be the image of $\bigcup \partial M_i$ in $M^n$.  Let $U$ be a regular neighborhood of $T$.  By Theorem \ref{t:OriginalIntersection}, we can isotope $S^n$ to a surface $S'$ such that each component of $S' \setminus U$ is a topologically minimal surface and the indices of these surfaces sum to at most $k$.  

Suppose some loop of $S' \cap \partial \overline{U}$ is trivial in $\partial \overline{U}$. Then an innermost such loop either bounds a disk in $S' \setminus U$, which can be isotoped into $U$, or  is parallel to a compressing disk for a component of $S'  \setminus U$.  In the latter case, this compressing disk is disjoint from all other compressing disks for $S' \setminus U$, which implies that the disk complex for $S' \setminus U$ is contractible.  This contradicts the fact that $S' \setminus U$ is topologically minimal.  Thus $S' \cap \partial \overline{U}$ must consist entirely of essential loops.

For each $i$, let $S'_i = S' \cap (M_i \setminus U)$. As $S'$ is separating, it must meet each component of $\partial \overline{U}$ in at least two loops. Thus, each surface $S'_i$ has at least four boundary components. 

Note that $S^n$ was a union of $n$ four-punctured spheres, each of which has Euler characteristic $-2$. Thus, the Euler characteristic of $S^n$, and thus $S'$, is $-2n$. If the Euler characteristic of $S'_i$ is not $-2$ for some $i$, then for some $S_j$, the Euler characteristic must be at least $-1$. But such a surface that has at least four boundary components must have at least one planar component. Since, by Theorem \ref{t:OriginalIntersection}, each component of $S'_i$ is topologically minimal, we thus violate Claim \ref{c:EulerChar}. 

We conclude that for each $i$, $\chi(S'_i)=-2$. By Theorem \ref{t:OriginalIntersection}, the sum of the indices of the components of $S' \setminus U$ is at most $k < n$. Thus, there is some $i$ such that the topological index of $S'_i$ is zero. Because the boundary of $M$ consists of tori, every boundary compressible surface is either compressible or a boundary parallel annulus.  The surface $S'_i$ cannot be a boundary parallel annulus, so every component of $S'_i$ must be essential. Since $S'_i$ has at least four boundary components, it must have a planar component whose Euler characteristic is at least $-2$. This violates our assumption that $L$ was chosen so that $M$ does not contain any such surfaces. This contradiction implies that $S^n$ must have index at least $n$.
\end{proof}

We have thus established that the topological index of $S^n$ is precisely $n$.  Since $n$ was an arbitrary integer, this proves Theorem \ref{t:ExistenceTheorem}.

\bibliographystyle{alpha}

\end{document}